\documentclass[11pt]{article}

\usepackage[a4paper,margin=1.05in]{geometry}
\usepackage{amsmath,amssymb,amsthm,mathtools}
\usepackage[T1]{fontenc}
\usepackage{libertinus}
\usepackage{microtype}
\usepackage{tikz}
\usepackage{float}
\usetikzlibrary{patterns,patterns.meta}
\usepackage{xcolor}
\usepackage[colorlinks=true,linkcolor=black,citecolor=black,urlcolor=black]{hyperref}

\newcommand{\K}{\mathbb K}

\newcommand{\N}{\mathbb N}
\newcommand{\eps}{\varepsilon}
\newcommand{\whotimes}{\widehat\otimes}

\newtheorem{theorem}{Theorem}[section]
\newtheorem{maintheorem}{Theorem}

\newtheorem{lemma}[theorem]{Lemma}
\newtheorem{proposition}[theorem]{Proposition}
\newtheorem{corollary}[theorem]{Corollary}

\title{The phase diagram of injective-to-projective tensor distortion}
\author{Daniel N\'u\~nez-Alarc\'on \and Daniel Marinho Pellegrino \and Joedson Silva dos Santos}
\date{}

\begin{document}
\maketitle
\begin{abstract}
For finite-dimensional Banach spaces $E$ and $F$, set
\[
 \rho(E,F):=\sup_{0\neq z\in E\otimes F}\frac{\pi(z)}{\varepsilon(z)}.
\]
The square growth of $\rho(\ell_p^d,\ell_q^d)$ was determined by Bonet, Defant, Peris and Ramanujan. We study the rectangular problem with the two dimensions varying independently. If $r=\min\{n,m\}$ and
$\eta(t)=\min\{1/t,1/t'\}$, then, whenever $p$ and $q$ lie on the same side of $2$,
\[
 \rho(\ell_p^n,\ell_q^m)\asymp_{p,q}
 r^{1/2}\min\{n^{\eta(p)},m^{\eta(q)}\}.
\]
Thus the classical square exponent splits into two dimensional scales.
In the mixed range $1\le p\le2\le q\le\infty$, writing $a=1/p$ and $b=1/q$, we obtain the lower estimate
\[
 \rho_{p,q}(n,m)\gtrsim_{p,q}
 \max\!\left\{r^{\min\{a+b,2-a-b\}},
 m^{b-1/2}r^{1/2}\min\{n^{1-a},m^{1/2}\}\right\},
\]
and the upper estimate
\[
 \rho_{p,q}(n,m)\lesssim_{p,q}
 \min\!\left\{n^{1-a}r^{1-b},m^b r^a,r\right\}.
\]
Moreover, if \((n-m)(a+b-1)\le0\), then
\[
\rho_{p,q}(n,m)\asymp_{p,q}
r^{\min\{a+b,\,2-a-b\}}.
\]

In the interior mixed range $1<p<2<q<p'$, if
\[
 \alpha_{p,q}:=
 \frac{(a+b-1)(\frac12-b)}{a-\frac12},
\]
then
\[
 \rho_{p,q}(n,m)\lesssim_{p,q}m^{1-\alpha_{p,q}},
 \qquad
 \sup_{n\ge1}\rho_{p,q}(n,m)\asymp_{p,q}m^{1-\alpha_{p,q}}.
\]
The corresponding statements in the reversed mixed range follow by duality and symmetry.
\end{abstract}

\medskip
\noindent\textbf{Keywords.} Injective tensor norm; projective tensor norm; finite-dimensional $\ell_p$ spaces; random matrices; absolutely summing operators; nuclear norm.

\noindent\textbf{2020 Mathematics Subject Classification.} 46B28; 46B45; 46G25; 47B10.

\begingroup
\small
\setcounter{tocdepth}{1}
\tableofcontents
\endgroup

\section{Introduction}

Let $\K\in\{\mathbb R,\mathbb C\}$. For finite-dimensional Banach spaces $E$ and $F$, let $\eps_{E,F}$ and $\pi_{E,F}$ denote the injective and projective norms on $E\otimes F$. We write
\[
 \rho(E,F):=\sup_{0\neq z\in E\otimes F}\frac{\pi(z)}{\eps(z)}
 =\left\|I_{E,F}:E\whotimes_{\eps}F\longrightarrow E\whotimes_{\pi}F\right\|.
\]
For $1\le p,q\le\infty$ and $n,m\in\N$, set
\[
 \rho_{p,q}(n,m):=\rho(\ell_p^n,\ell_q^m),
 \qquad
 \rho_{p,q}(d):=\rho_{p,q}(d,d).
\]
Equivalently, $\rho(E,F)$ is the optimal comparison between the nuclear norm and the operator norm of finite-rank maps $E^*\to F$.

The comparison between the injective and projective tensor norms belongs to the classical metric theory of tensor products and operator ideals; standard references include Defant--Floret \cite{DefantFloret} and Ryan \cite{Ryan}, while Pisier \cite{PisierGT} places these questions in the broader Grothendieck-theoretic setting.  For the specific $\eps$--$\pi$ comparison on sequence spaces, Defant and Mascioni \cite{DefantMascioni1990} introduced a limit-order theory for $\eps$--$\pi$ continuous operators and computed the relevant diagonal orders.  Building on this circle of ideas, Bonet, Defant, Peris and Ramanujan \cite[Proposition~3 and the remark following it]{BonetDefantPerisRamanujan1994} determined the complete square growth
\[
 \rho_{p,q}(d)\asymp_{p,q}d^{1/\mu(p,q)},
\]
where
\[
 \frac1{\mu(p,q)}=
 \begin{cases}
 \frac32-\max\{1/p,1/q\},&1\le p,q\le2,\\[2mm]
 \frac12+\min\{1/p,1/q\},&2\le p,q\le\infty,\\[2mm]
 \min\{1/p+1/q,\,2-1/p-1/q\},&\text{otherwise}.
 \end{cases}
\]
Thus the square exponent was completely determined in 1994.  Rectangular tensor-product identities with two independent dimensions also occur in the work of Defant and Michels \cite{DefantMichels2006}, who obtained asymptotically optimal estimates for identities from $E^n\whotimes_{\eps}F^m$ into the Hilbert space $\ell_2^{nm}$, together with applications to geometric parameters of tensor products.  Here the target is instead the projective norm on the same rectangular tensor product.  Gupta, Misra and Ray \cite{GuptaMisraRay} later returned to closely related tensor-product norm ratios in connection with variants of Grothendieck's inequality and obtained, among other results, diagonal and endpoint estimates.  Aubrun, Lami, Palazuelos, Szarek and Winter \cite{AubrunEtAl2020} placed tensor-norm ratios in a different, operational setting: they related them to maximal global-to-local bias advantages in XOR games in generalized probabilistic theories and, in particular, identified the endpoint $\rho(\ell_1^n,X)$ with a finite $1$-summing norm.  A distinct asymptotic direction, in which tensor powers rather than the two dimensions vary independently, has been developed by Aubrun and M\"uller-Hermes \cite{AubrunMullerHermes2025,AubrunMullerHermes2026}.

The rectangular problem begins when the two dimensions are allowed to grow independently. We therefore put
\[
 r:=\min\{n,m\},
 \qquad
 \eta(t):=\min\left\{\frac1t,\frac1{t'}\right\}.
\]
The rectangular growth is completely determined when $p$ and $q$ lie on the same side of $2$.

\begin{maintheorem}[Same-side rectangular formula]\label{thm:A}
Let $1\le p,q\le\infty$ and $n,m\in\N$. Assume either
\[
 1\le p,q\le2
 \qquad\text{or}\qquad
 2\le p,q\le\infty.
\]
Then
\[
 \rho_{p,q}(n,m)
 \asymp_{p,q}
 r^{1/2}\min\{n^{\eta(p)},m^{\eta(q)}\}.
\]
The comparison constants depend only on $p,q$ and the scalar field. In the Hilbertian case,
\[
 \rho_{2,2}(n,m)=r.
\]
\end{maintheorem}

When \(n=m\), Theorem~\ref{thm:A} reduces to the square formula. Indeed, if
\(n=m=d\), then \(r=d\) and the theorem gives
\[
\rho_{p,q}(d,d)
\asymp_{p,q}
d^{1/2+\min\{\eta(p),\eta(q)\}}.
\]
For \(1\le p,q\le2\), this exponent is
\[
\frac32-\max\left\{\frac1p,\frac1q\right\},
\]
and for \(2\le p,q\le\infty\), it is
\[
\frac12+\min\left\{\frac1p,\frac1q\right\}.
\]
These are the square exponents obtained by Bonet, Defant, Peris and
Ramanujan \cite{BonetDefantPerisRamanujan1994}. Theorem~\ref{thm:A} determines the dependence on the two dimensions separately.

For $p=q$ this becomes
\[
 \rho(\ell_p^n,\ell_p^m)
 \asymp_p
 r^{\frac12+\eta(p)},
\]
which recovers the square exponent when $n=m$. The rectangular feature of Theorem~\ref{thm:A} is the interaction of the two scales
\[
 r^{1/2}n^{\eta(p)}
 \qquad\text{and}\qquad
 r^{1/2}m^{\eta(q)}.
\]
The square theorem gives the correct lower scale on the diagonal coordinate copy of dimension $r$, while a genuinely rectangular argument is needed for an upper estimate uniform in the larger dimension.

The proof below $2$ uses the following rectangular endpoint, a quantitative form of the finite $1$-summing identity of \cite[Proposition~13]{AubrunEtAl2020}:
\[
 \rho(\ell_1^n,\ell_q^m)\asymp_q\sqrt{\min\{n,m\}},
 \qquad 1\le q\le2.
\]
The two rectangular upper bounds at this endpoint follow from direct Khintchine estimates, while the matching lower estimate in Theorem~\ref{thm:A} is obtained from a random sign matrix with two simultaneous operator-norm controls. The range above $2$ follows by tensor duality.

The mixed range has a different geometry. The next theorem gives sharp polynomial
orders on an explicit part of that range and two-sided estimates elsewhere.

\begin{maintheorem}[Mixed rectangular bounds and sharp-order region]\label{thm:B}
Let $1\le p\le2\le q\le\infty$, put
\[
 a:=\frac1p,\qquad b:=\frac1q,\qquad r:=\min\{n,m\},
 \qquad \gamma_{\rm mix}:=\min\{a+b,2-a-b\}.
\]
Then
\[
 \max\!\left\{
 r^{\gamma_{\rm mix}},\,
 m^{b-1/2}r^{1/2}\min\{n^{1-a},m^{1/2}\}
 \right\}
 \lesssim_{p,q}\rho_{p,q}(n,m)
 \lesssim_{p,q}
 \min\left\{n^{1-a}r^{1-b},\,m^b r^a,\,r\right\}.
\]
If
\[
 (n-m)(a+b-1)\le0,
\]
then the sharp polynomial order is determined:
\[
 \rho_{p,q}(n,m)=r^{\gamma_{\rm mix}}.
\]
If, in addition, $1<p<2<q<p'$, define
\[
 \alpha_{p,q}:=
 \frac{(a+b-1)(\frac12-b)}{a-\frac12}.
\]
Then
\[
 \rho_{p,q}(n,m)\lesssim_{p,q}m^{1-\alpha_{p,q}},
 \qquad
 \sup_{n\ge1}\rho_{p,q}(n,m)\asymp_{p,q}m^{1-\alpha_{p,q}}.
\]
By duality and symmetry the corresponding statement holds when $1\le q\le2\le p\le\infty$.
\end{maintheorem}

The condition
\[
(n-m)(a+b-1)\le0
\]
has a simple interpretation. If \(a+b>1\), the sharp polynomial order
is obtained when \(n\le m\). If \(a+b<1\), it is obtained when
\(n\ge m\). On the line \(a+b=1\), the condition is automatic; in that
case
\[
\rho(\ell_p^n,\ell_{p'}^m)=\min\{n,m\}.
\]

\definecolor{ssblue}{RGB}{218,232,245}
\definecolor{exactgreen}{RGB}{218,238,225}
\definecolor{openorange}{RGB}{249,229,204}
\begin{figure}[H]
\centering
\begin{tikzpicture}[x=4.25cm,y=4.25cm,every node/.style={font=\scriptsize}]
\newcommand{\phasepanel}[1]{%
  \fill[ssblue] (0,0) rectangle (.5,.5);
  \fill[ssblue] (.5,.5) rectangle (1,1);
  \ifnum#1=1
    \fill[exactgreen] (.5,0)--(.5,.5)--(1,0)--cycle;
    \fill[openorange] (.5,.5)--(1,.5)--(1,0)--cycle;
    \fill[openorange] (0,.5)--(.5,.5)--(0,1)--cycle;
    \fill[exactgreen] (0,1)--(.5,.5)--(.5,1)--cycle;
  \else
    \fill[openorange] (.5,0)--(.5,.5)--(1,0)--cycle;
    \fill[exactgreen] (.5,.5)--(1,.5)--(1,0)--cycle;
    \fill[exactgreen] (0,.5)--(.5,.5)--(0,1)--cycle;
    \fill[openorange] (0,1)--(.5,.5)--(.5,1)--cycle;
  \fi
  \draw[black!70] (0,0) rectangle (1,1);
  \draw[black!45] (.5,0)--(.5,1);
  \draw[black!45] (0,.5)--(1,.5);
  \draw[dashed,black!55] (.5,.5)--(1,0);
  \draw[dashed,black!55] (0,1)--(.5,.5);
  \node[below=2pt] at (0,0) {$0$};
  \node[below=2pt] at (.5,0) {$1/2$};
  \node[below=2pt] at (1,0) {$1$};
  \node[left=2pt] at (0,.5) {$1/2$};
  \node[left=2pt] at (0,1) {$1$};
  \node[text=black!70] at (.25,.25) {Theorem A};
  \node[text=black!70] at (.75,.75) {Theorem A};
  \node[fill=white,fill opacity=.86,text opacity=1,inner sep=1pt,
        anchor=south west] at (.535,.455) {$u+v=1$};
}
\begin{scope}
  \phasepanel{1}
  \node[font=\small] at (.5,1.10) {$n\ge m$};
  \node[below=14pt,font=\footnotesize] at (.5,0) {$u=1/p$};
  \node[rotate=90,above=15pt,font=\footnotesize] at (0,.5) {$v=1/q$};
\end{scope}
\begin{scope}[xshift=5.35cm]
  \phasepanel{2}
  \node[font=\small] at (.5,1.10) {$n\le m$};
  \node[below=14pt,font=\footnotesize] at (.5,0) {$u=1/p$};
  \node[rotate=90,above=15pt,font=\footnotesize] at (0,.5) {$v=1/q$};
\end{scope}
\end{tikzpicture}

\vspace{4mm}
\begin{tikzpicture}[every node/.style={font=\scriptsize},x=1cm,y=1cm]
  \fill[ssblue] (0,1.00) rectangle (.34,1.15); \draw[black!55] (0,1.00) rectangle (.34,1.15);
  \node[anchor=west] at (.48,1.075) {same-side region: complete formula (Theorem~\ref{thm:A})};
  \fill[exactgreen] (0,.50) rectangle (.34,.65); \draw[black!55] (0,.50) rectangle (.34,.65);
  \node[anchor=west] at (.48,.575) {mixed region: sharp polynomial order (Theorem~\ref{thm:B})};
  \fill[openorange] (0,0) rectangle (.34,.15); \draw[black!55] (0,0) rectangle (.34,.15);
  \node[anchor=west] at (.48,.075) {mixed region: two-sided bounds; sharp order open};
\end{tikzpicture}

\caption{Rectangular phase diagram in reciprocal coordinates $u=1/p$ and
$v=1/q$. Both mixed quadrants are displayed. Blue marks the same-side regions
completely determined by Theorem~\ref{thm:A}; green marks the mixed regions where Theorem~\ref{thm:B} determines
the sharp polynomial order; orange marks the remaining mixed regions,
where the present results give two-sided bounds. The dashed segments are the
transition line $u+v=1$. Reversing the order of $n$ and $m$ interchanges the
green and orange mixed triangles.}
\label{fig:phase-rectangular}
\end{figure}
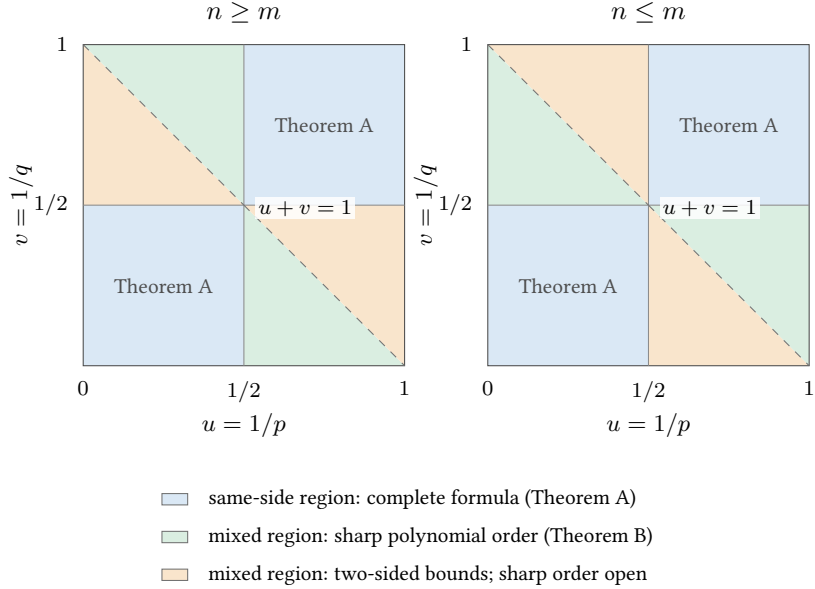

On the transition line $1/p+1/q=1$, Theorem~\ref{thm:B} reaches the universal dimension bound and gives, for all $n,m$,
\[
 \rho(\ell_p^n,\ell_{p'}^m)=\min\{n,m\}.
\]
For $n=m=d$, it yields the exact identity
\[
 \rho_{p,q}(d,d)=d^{\min\{1/p+1/q,\,2-1/p-1/q\}}.
\]
Thus the classical mixed square exponent is recovered in every finite
dimension, up to constants depending only on the exponents. The same
sharp polynomial order also holds on an entire nonsquare region.

The operator interpretation of $\rho$ converts Theorem~\ref{thm:A} into a complete same-side rectangular estimate for the ratio between nuclear and operator norms.  In the real case, the tensor-ratio identity of Aubrun, Lami, Palazuelos, Szarek and Winter transfers these estimates to asymmetric XOR games in generalized probabilistic theories; see Section~\ref{sec:consequences}.  Here the two local dimensions remain independent, so the rectangular estimates yield information not visible in the square specialization.

Throughout, $A\lesssim_{p,q}B$ means $A\le C_{p,q}B$ with a constant independent of the dimensions, and $A\asymp_{p,q}B$ means that both inequalities hold.

\section{Tensorial preliminaries}\label{sec:prelim}

For $1\le t\le\infty$, let $t'$ be the conjugate exponent. The canonical basis of $\K^k$ is denoted by $(e_i)_{i=1}^k$.

For finite-dimensional Banach spaces $E$ and $F$, the canonical pairing is determined by
\[
 \langle x^*\otimes y^*,x\otimes y\rangle=x^*(x)y^*(y).
\]
The standard duality identifications are
\[
 (E\whotimes_{\pi}F)^*\cong E^*\whotimes_{\eps}F^*,
 \qquad
 (E\whotimes_{\eps}F)^*\cong E^*\whotimes_{\pi}F^*.
\]

\begin{lemma}\label{lem:duality}
For finite-dimensional Banach spaces $E$ and $F$,
\[
 \rho(E,F)=\rho(E^*,F^*)=\rho(F,E).
\]
\end{lemma}

\begin{proof}
The flip map $x\otimes y\mapsto y\otimes x$ is an isometry for both $\eps$ and $\pi$, which gives $\rho(E,F)=\rho(F,E)$. Let
\[
 I_{E,F}:E\whotimes_{\eps}F\longrightarrow E\whotimes_{\pi}F
\]
be the canonical identity. Its adjoint has the same norm. Under the duality identifications above, $I_{E,F}^*$ is precisely
\[
 I_{E^*,F^*}:E^*\whotimes_{\eps}F^*\longrightarrow E^*\whotimes_{\pi}F^*.
\]
Hence $\rho(E,F)=\|I_{E,F}\|=\|I_{E,F}^*\|=\rho(E^*,F^*)$.
\end{proof}

\begin{lemma}[Dimension bound]\label{lem:dimension}
If $E$ and $F$ are finite-dimensional Banach spaces, then
\[
 \rho(E,F)\le \min\{\dim E,\dim F\}.
\]
\end{lemma}

\begin{proof}
Let $N=\dim E$ and choose an Auerbach basis $(x_i,x_i^*)_{i=1}^N$ for $E$. Every $z\in E\otimes F$ has the representation
\[
 z=\sum_{i=1}^N x_i\otimes y_i,
 \qquad
 y_i=(x_i^*\otimes I_F)z.
\]
Since $\|x_i\|=\|x_i^*\|=1$,
\[
 \|y_i\|\le\eps(z)
\]
for every $i$. Hence
\[
 \pi(z)\le\sum_{i=1}^N\|y_i\|\le N\eps(z).
\]
Applying the same argument after flipping the tensor factors gives
$\pi(z)\le(\dim F)\eps(z)$.
\end{proof}

\begin{lemma}\label{lem:isomorphism}
Let $E,E_0,F$ be finite-dimensional Banach spaces and let $S:E\to E_0$ be an isomorphism. Then
\[
 \rho(E,F)\le \|S\|\,\|S^{-1}\|\,\rho(E_0,F).
\]
The same statement holds when the second tensor factor is changed.
\end{lemma}

\begin{proof}
Fix $z\in E\otimes F$ and put $w=(S\otimes I_F)z$. The metric mapping property gives
\[
 \pi_{E,F}(z)\le \|S^{-1}\|\pi_{E_0,F}(w),
 \qquad
 \eps_{E_0,F}(w)\le \|S\|\eps_{E,F}(z).
\]
Therefore
\[
 \frac{\pi_{E,F}(z)}{\eps_{E,F}(z)}
 \le
 \|S\|\,\|S^{-1}\|
 \frac{\pi_{E_0,F}(w)}{\eps_{E_0,F}(w)}.
\]
Since $S\otimes I_F$ is an isomorphism on the algebraic tensor product, taking suprema proves the claim.
\end{proof}

For $1\le u,v\le\infty$, the canonical identity satisfies
\begin{equation}\label{eq:identity-lp}
 \|I_{u,v}^{(d)}:\ell_u^d\to\ell_v^d\|
 =d^{\max\{0,1/v-1/u\}}.
\end{equation}
We also use
\begin{equation}\label{eq:l1pi}
 \ell_1^d\whotimes_{\pi}X\cong\ell_1^d(X)
\end{equation}
isometrically.

\section{Same-side rectangular distortion}\label{sec:sameside}

We first isolate the endpoint that supplies the dimension-sensitive upper bound below $2$.

\begin{proposition}[The $\ell_1$ endpoint below $2$]\label{prop:l1below}
Let $1\le q\le2$, $n,m\in\N$, and $r=\min\{n,m\}$. Then
\[
 \rho(\ell_1^n,\ell_q^m)\asymp_q\sqrt r.
\]
The constants are independent of $n$ and $m$.
\end{proposition}

\begin{proof}
Write
\[
 z=\sum_{i=1}^n e_i\otimes y_i\in\ell_1^n\otimes\ell_q^m.
\]
By \eqref{eq:l1pi},
\[
 \pi(z)=\sum_{i=1}^n\|y_i\|_q,
 \qquad
 \eps(z)=\sup_{\varphi\in B_{\ell_{q'}^m}}
 \sum_{i=1}^n|\varphi(y_i)|.
\]
Equivalently,
\begin{equation}\label{eq:finite-1summing}
 \rho(\ell_1^n,X)=\pi_1^{(n)}(I_X),
\end{equation}
where $\pi_1^{(n)}$ denotes the $1$-summing norm restricted to sequences of length $n$; compare \cite[Proposition~13]{AubrunEtAl2020}.

We first obtain the bound in terms of $n$. By Cauchy--Schwarz,
\[
 \pi(z)\le \sqrt n\left(\sum_{i=1}^n\|y_i\|_q^2\right)^{1/2}.
\]
Since $1\le q\le2$, the mixed-norm inequality and the scalar Khintchine inequality give
\begin{align*}
 \left(\sum_{i=1}^n\|y_i\|_q^2\right)^{1/2}
 &\le
 \left(\sum_{j=1}^m
       \left(\sum_{i=1}^n|y_i(j)|^2\right)^{q/2}
 \right)^{1/q}\\
 &\lesssim_q
 \left(\mathbb E\left\|\sum_{i=1}^n\varepsilon_i y_i\right\|_q^q\right)^{1/q}
 \le \eps(z),
\end{align*}
where $(\varepsilon_i)$ are independent Rademacher variables. Hence
\begin{equation}\label{eq:l1below-n}
 \rho(\ell_1^n,\ell_q^m)\lesssim_q\sqrt n.
\end{equation}

For the bound in terms of $m$, fix independent Rademacher variables
$(\delta_j)_{j=1}^m$. By the scalar Khintchine inequality,
\[
 \sum_{i=1}^n\|y_i\|_2
 \lesssim
 \mathbb E_\delta
 \sum_{i=1}^n
 \left|\sum_{j=1}^m\delta_j y_i(j)\right|.
\]
For every choice of signs, the functional
\[
 \varphi_\delta
 =m^{-1/q'}(\delta_1,\ldots,\delta_m)
 \]
belongs to $B_{\ell_{q'}^m}$. Hence, by the definition of the injective norm,
\[
 \sum_{i=1}^n
 \left|\sum_{j=1}^m\delta_j y_i(j)\right|
 =m^{1/q'}\sum_{i=1}^n|\varphi_\delta(y_i)|
 \le m^{1/q'}\eps(z).
\]
Since $1\le q\le2$,
\[
 \|y_i\|_q\le m^{1/q-1/2}\|y_i\|_2,
\]
and therefore
\[
 \pi(z)=\sum_{i=1}^n\|y_i\|_q
 \lesssim_q
 m^{1/q-1/2+1/q'}\eps(z)
 =\sqrt m\,\eps(z).
\]
Thus
\begin{equation}\label{eq:l1below-m}
 \rho(\ell_1^n,\ell_q^m)\lesssim_q\sqrt m.
\end{equation}
Combining \eqref{eq:l1below-n} and \eqref{eq:l1below-m} gives the required upper bound $\lesssim_q\sqrt r$.

For the reverse estimate, the first $r$ coordinates are $1$-complemented in both factors, so
\[
 \rho(\ell_1^n,\ell_q^m)\ge \rho(\ell_1^r,\ell_q^r).
\]
The square estimate of Bonet, Defant, Peris and Ramanujan \cite{BonetDefantPerisRamanujan1994} has exponent $1/2$ at $(1,q)$ for $1\le q\le2$. Hence
\[
 \rho(\ell_1^r,\ell_q^r)\asymp_q\sqrt r,
\]
and the proof is complete.
\end{proof}

For the lower estimate in the general same-side problem, we use one
sign matrix with two simultaneous norm controls. We fix the following
convention. If
\[
A=(a_{ji})_{1\le j\le m,\ 1\le i\le n},
\]
then
\[
z_A:=\sum_{i=1}^n\sum_{j=1}^m a_{ji}e_i\otimes e_j
\in\ell_p^n\otimes\ell_q^m.
\]
Under the injective tensor/operator identification, \(z_A\) corresponds
to
\[
A:\ell_{p'}^n\longrightarrow\ell_q^m,
\qquad
(Ax)_j=\sum_{i=1}^n a_{ji}x_i.
\]
The same coefficients define the bilinear form
\[
B_A:\ell_p^n\times\ell_{q'}^m\longrightarrow\K,
\qquad
B_A(x,y)=\sum_{i=1}^n\sum_{j=1}^m a_{ji}x_i y_j.
\]
Thus
\[
\eps(z_A)=\|A:\ell_{p'}^n\to\ell_q^m\|.
\]
When \(a_{ji}\in\{-1,1\}\), one also has
\[
\langle B_A,z_A\rangle=nm.
\]

\begin{lemma}\label{lem:randommatrix}
Let $1\le p,q\le2$ and $n,m\in\N$. There is a sign matrix
$A=(\sigma_{ji})_{1\le j\le m,\,1\le i\le n}$ such that
\begin{align}
\|A:\ell_{p'}^n\to\ell_q^m\|
&\lesssim_{p,q}
n^{1/p-1/2}m^{1/q-1/2}(\sqrt n+\sqrt m),
\label{eq:rand-eps}\\
\|A:\ell_p^n\to\ell_{q'}^m\|
&\lesssim_{p,q}
n^{1-1/p}+m^{1-1/q}.
\label{eq:rand-bilin}
\end{align}
\end{lemma}

\begin{proof}
Let \(A=(\sigma_{ji})\) be an \(m\times n\) random matrix with
independent Rademacher entries. We apply
\cite[Equation~(2)]{LatalaStrzeleckaJFA2025} twice.

First, take the domain exponent to be \(p'\) and the target exponent to
be \(q\). Since \(p'\ge2\) and \(q\le2\), the corresponding estimate
gives
\[
\mathbb E\|A:\ell_{p'}^n\to\ell_q^m\|
\lesssim_{p,q}
n^{1/p-1/2}m^{1/q-1/2}(\sqrt n+\sqrt m).
\]
Second, take the domain exponent to be \(p\) and the target exponent to
be \(q'\). Since \(p\le2\) and \(q'\ge2\), one obtains
\[
\mathbb E\|A:\ell_p^n\to\ell_{q'}^m\|
\lesssim_{p,q}
n^{1-1/p}+m^{1-1/q}.
\]

For complex scalars, the same real sign matrix acts complex-linearly and changes the bounds by at most an absolute factor. Markov's inequality then provides one realization for which both estimates hold simultaneously.
\end{proof}

For $1\le p,q\le2$, put
\[
 \beta_p:=1-\frac1p,
 \qquad
 \beta_q:=1-\frac1q.
\]

\begin{lemma}\label{lem:lower}
Let $1\le p,q\le2$. Then
\[
 \rho_{p,q}(n,m)
 \gtrsim_{p,q}
 \sqrt{\min\{n,m\}}\,
 \min\{n^{\beta_p},m^{\beta_q}\}.
\]
\end{lemma}

\begin{proof}
Choose the sign matrix $A=(\sigma_{ji})$ from Lemma~\ref{lem:randommatrix} and set
\[
 z_A=\sum_{i=1}^n\sum_{j=1}^m\sigma_{ji}e_i\otimes e_j.
\]
The injective tensor/operator identification and
\eqref{eq:rand-eps} give
\[
 \eps(z_A)\lesssim_{p,q}
 n^{1/p-1/2}m^{1/q-1/2}(\sqrt n+\sqrt m).
\]
The same matrix defines the bilinear form
\[
 B_A(x,y)=\sum_{i=1}^n\sum_{j=1}^m\sigma_{ji}x_i y_j,
\]
and \eqref{eq:rand-bilin} gives
\[
 \|B_A\|\lesssim_{p,q}n^{\beta_p}+m^{\beta_q}.
\]
Since $\langle B_A,z_A\rangle=nm$, projective duality yields
\[
 \pi(z_A)\gtrsim_{p,q}\frac{nm}{n^{\beta_p}+m^{\beta_q}}.
\]
Therefore
\[
\rho_{p,q}(n,m)
\gtrsim_{p,q}
\frac{
n^{3/2-1/p}m^{3/2-1/q}
}{
(\sqrt n+\sqrt m)
\bigl(n^{1-1/p}+m^{1-1/q}\bigr)
}.
\]
Put
\[
\beta_p=1-\frac1p,
\qquad
\beta_q=1-\frac1q.
\]
We simplify the last expression by considering the two possible orders of
\(n\) and \(m\).

Assume first that \(n\le m\). Then
\[
\sqrt n+\sqrt m\asymp\sqrt m,
\]
and hence
\[
\rho_{p,q}(n,m)
\gtrsim_{p,q}
\frac{
n^{1/2+\beta_p}m^{\beta_q}
}{
n^{\beta_p}+m^{\beta_q}
}.
\]
Since
\[
\frac{uv}{u+v}\asymp\min\{u,v\}
\qquad (u,v>0),
\]
we obtain
\[
\rho_{p,q}(n,m)
\gtrsim_{p,q}
\sqrt n\,
\min\{n^{\beta_p},m^{\beta_q}\}.
\]
Since \(r=n\) in this case, this is the required estimate.

If \(m\le n\), the same argument, with the two tensor factors
interchanged, gives
\[
\rho_{p,q}(n,m)
\gtrsim_{p,q}
\sqrt m\,
\min\{n^{\beta_p},m^{\beta_q}\}.
\]
Combining the two cases proves
\[
\rho_{p,q}(n,m)
\gtrsim_{p,q}
\sqrt r\,
\min\{n^{\beta_p},m^{\beta_q}\}.
\]
\end{proof}

\begin{proof}[Proof of Theorem~\ref{thm:A}]
Assume first $1\le p,q\le2$ and write $r=\min\{n,m\}$. Lemma~\ref{lem:lower} gives
\[
 \rho_{p,q}(n,m)
 \gtrsim_{p,q}
 \sqrt r\min\{n^{\beta_p},m^{\beta_q}\}.
\]
For the reverse estimate, Lemma~\ref{lem:isomorphism} and Proposition~\ref{prop:l1below} give
\[
 \rho_{p,q}(n,m)
 \le n^{1-1/p}\rho(\ell_1^n,\ell_q^m)
 \lesssim_{p,q}n^{\beta_p}\sqrt r.
\]
Interchanging the two tensor factors gives
\[
 \rho_{p,q}(n,m)\lesssim_{p,q}m^{\beta_q}\sqrt r.
\]
Thus
\[
 \rho_{p,q}(n,m)
 \asymp_{p,q}
 \sqrt r\min\{n^{\beta_p},m^{\beta_q}\}.
\]
Since $\eta(t)=1-1/t$ for $1\le t\le2$, this is the asserted formula.

If $2\le p,q\le\infty$, Lemma~\ref{lem:duality} gives
\[
 \rho_{p,q}(n,m)=\rho_{p',q'}(n,m).
\]
Now $1\le p',q'\le2$ and $\eta(p')=\eta(p)$, $\eta(q')=\eta(q)$, so the first part applies.

Finally, when $p=q=2$, tensors identify with operators of rank at most $r$. Their injective and projective norms are the operator and nuclear norms, respectively, so
\[
 \pi(z)\le \operatorname{rank}(z)\eps(z)\le r\eps(z).
\]
A partial identity of rank $r$ attains equality. Hence $\rho_{2,2}(n,m)=r$.
\end{proof}

\section{The mixed rectangular range}\label{sec:mixed}

We begin with a rectangular endpoint estimate. This endpoint drives the sharp-order region in Theorem~\ref{thm:B}.

\begin{proposition}[The $\ell_1$ endpoint]\label{prop:l1q}
Let $2\le q\le\infty$, $n,m\in\N$, and $r=\min\{n,m\}$. Then
\[
 \rho(\ell_1^n,\ell_q^m)\asymp_q r^{1/q'}.
\]
If $n\le m$, then the equality is exact:
\[
 \rho(\ell_1^n,\ell_q^m)=n^{1/q'}.
\]
For $q=\infty$ one has the exact identity
\[
 \rho(\ell_1^n,\ell_\infty^m)=\min\{n,m\}.
\]
\end{proposition}

\begin{proof}
Write
\[
 z=\sum_{i=1}^n e_i\otimes y_i\in\ell_1^n\otimes\ell_q^m.
\]
By \eqref{eq:l1pi},
\[
 \pi(z)=\sum_{i=1}^n\|y_i\|_q,
 \qquad
 \eps(z)=\sup_{\varphi\in B_{\ell_{q'}^m}}
 \sum_{i=1}^n|\varphi(y_i)|.
\]
Assume first $q<\infty$. If $(\eta_i)_{i=1}^n$ are independent Rademacher variables, then
\[
 \sum_{i=1}^n\|y_i\|_q^q
 \le \mathbb E\left\|\sum_{i=1}^n\eta_i y_i\right\|_q^q
 \le \eps(z)^q.
\]
Indeed, for each coordinate the $L_q$-norm of the Rademacher sum dominates its $L_2$-norm, and $q\ge2$. Hence
\[
 \pi(z)
 \le n^{1/q'}\left(\sum_{i=1}^n\|y_i\|_q^q\right)^{1/q}
 \le n^{1/q'}\eps(z).
\]
This already proves the sharp upper estimate when $n\le m$.

For the estimate in terms of $m$, let $(\varepsilon_j)_{j=1}^m$ be independent Rademacher variables and set
\[
 \varphi_\varepsilon=m^{-1/q'}(\varepsilon_1,\ldots,\varepsilon_m)\in B_{\ell_{q'}^m}.
\]
The scalar Khintchine inequality gives a constant $c>0$, independent of the dimensions, such that
\begin{align*}
 \eps(z)
 &\ge \mathbb E\sum_{i=1}^n|\varphi_\varepsilon(y_i)|\\
 &\ge c\,m^{-1/q'}\sum_{i=1}^n\|y_i\|_2
 \ge c\,m^{-1/q'}\sum_{i=1}^n\|y_i\|_q.
\end{align*}
Thus $\pi(z)\lesssim_q m^{1/q'}\eps(z)$, and therefore
\[
 \rho(\ell_1^n,\ell_q^m)\lesssim_q r^{1/q'}.
\]

If $q=\infty$, Lemma~\ref{lem:dimension} gives directly
\[
 \rho(\ell_1^n,\ell_\infty^m)\le r.
\]

For every $2\le q\le\infty$, take the diagonal tensor supported on the first $r$ coordinates,
\[
 z_r=\sum_{i=1}^r e_i\otimes e_i.
\]
Then $\pi(z_r)=r$ and $\eps(z_r)=r^{1/q}$, so
\[
 \frac{\pi(z_r)}{\eps(z_r)}=r^{1/q'}.
\]
This proves the lower bound, the sharp-order statement for $n\le m$, and the exact $q=\infty$ formula.
\end{proof}

\begin{corollary}\label{cor:pinfty}
Let $1\le p\le2$ and $n,m\in\N$. With $r=\min\{n,m\}$,
\[
 \rho(\ell_p^n,\ell_\infty^m)\asymp_p r^{1/p}.
\]
If $m\le n$, then
\[
 \rho(\ell_p^n,\ell_\infty^m)=m^{1/p}.
\]
\end{corollary}

\begin{proof}
By duality and the flip symmetry,
\[
 \rho(\ell_p^n,\ell_\infty^m)
 =\rho(\ell_1^m,\ell_{p'}^n).
\]
Apply Proposition~\ref{prop:l1q} with $q=p'$.
\end{proof}

We next consider the full mixed range.

\begin{proposition}\label{prop:mixedrect}
Let $1\le p\le2\le q\le\infty$, put
\[
 a=\frac1p,\qquad b=\frac1q,
 \qquad r=\min\{n,m\},
 \qquad \gamma_{\rm mix}=\min\{a+b,2-a-b\}.
\]
Then
\[
 \max\!\left\{
 r^{\gamma_{\rm mix}},\,
 m^{b-1/2}r^{1/2}\min\{n^{1-a},m^{1/2}\}
 \right\}
 \lesssim_{p,q}\rho_{p,q}(n,m)
 \lesssim_{p,q}
 \min\left\{n^{1-a}r^{1-b},\,m^b r^a,\,r\right\}.
\]
\end{proposition}

\begin{proof}
For the first upper estimate, compare $\ell_p^n$ with $\ell_1^n$. Formula \eqref{eq:identity-lp} and Lemma~\ref{lem:isomorphism} give
\[
 \rho_{p,q}(n,m)
 \le n^{1-a}\rho(\ell_1^n,\ell_q^m)
 \lesssim_q n^{1-a}r^{1-b}.
\]
For the second estimate compare $\ell_q^m$ with $\ell_\infty^m$. Since
\[
 \|I_{q,\infty}^{(m)}\|=1,
 \qquad
 \|I_{\infty,q}^{(m)}\|=m^b,
\]
Corollary~\ref{cor:pinfty} yields
\[
 \rho_{p,q}(n,m)
 \le m^b\rho(\ell_p^n,\ell_\infty^m)
 \lesssim_p m^b r^a.
\]
Finally, Lemma~\ref{lem:dimension} gives $\rho_{p,q}(n,m)\le r$.

For the lower estimate, use the diagonal tensor
\[
 z_r=\sum_{j=1}^r e_j\otimes e_j.
\]
Under the injective tensor/operator identification,
\[
 \eps(z_r)
 =r^{\max\{0,a+b-1\}}.
\]
The diagonal bilinear form on the same $r$ coordinates has norm
\[
 r^{\max\{0,1-a-b\}},
\]
and its pairing with $z_r$ is $r$. Hence
\[
 \pi(z_r)\ge r^{1-\max\{0,1-a-b\}}.
\]
Dividing gives
\[
 \frac{\pi(z_r)}{\eps(z_r)}
 \ge r^{\min\{a+b,2-a-b\}}.
\]

For the second lower estimate, compare the second tensor factor with
$\ell_2^m$. Since
\[
 \|I_{q,2}^{(m)}\|=m^{1/2-b},
 \qquad
 \|I_{2,q}^{(m)}\|=1,
\]
Lemma~\ref{lem:isomorphism}, applied with the roles of $\ell_q^m$ and
$\ell_2^m$ reversed, gives
\[
 \rho_{p,q}(n,m)
 \ge m^{b-1/2}\rho_{p,2}(n,m).
\]
Theorem~\ref{thm:A} yields
\[
 \rho_{p,2}(n,m)
 \asymp_p r^{1/2}\min\{n^{1-a},m^{1/2}\},
\]
and hence
\[
 \rho_{p,q}(n,m)
 \gtrsim_{p,q}
 m^{b-1/2}r^{1/2}\min\{n^{1-a},m^{1/2}\}.
\]
Together with the diagonal estimate, this proves the asserted lower bound.
\end{proof}

\begin{proposition}[Saturation in the interior mixed range]\label{prop:saturation}
Let $1<p<2<q<p'<\infty$ and $m\in\N$. Put
\[
 a=\frac1p,\qquad b=\frac1q,
 \qquad
 \alpha_{p,q}:=
 \frac{(a+b-1)(\frac12-b)}{a-\frac12}.
\]
Then, uniformly in $n$,
\[
 \rho_{p,q}(n,m)\lesssim_{p,q}m^{1-\alpha_{p,q}},
\]
and the exponent is sharp in the saturation limit:
\[
 \sup_{n\ge1}\rho_{p,q}(n,m)
 \asymp_{p,q}m^{1-\alpha_{p,q}}.
\]
\end{proposition}

\begin{proof}
For an operator $U:X\to X$ and $1<P<\infty$, let
\[
 \gamma_P(U)
 :=\inf\{\|B\|\,\|A\|:
 U=BA,\ A:X\to L_P(\mu),\ B:L_P(\mu)\to X\},
\]
where the infimum runs over all measure spaces $(\Omega,\mu)$ and all such factorizations.

Fix $n$ and write $X=\ell_q^m$. By trace duality for the nuclear norm,
\[
 \rho_{p,q}(n,m)
 =\sup\bigl\{|\operatorname{tr}(TS)|:
 T:\ell_{p'}^n\to X,\ S:X\to\ell_{p'}^n,
 \ \|T\|\le1,\ \|S\|\le1\bigr\}.
\]
Every $U=TS$ in this supremum factors through the $L_{p'}$-space
$\ell_{p'}^n$, so $\gamma_{p'}(U)\le1$. Hence
\begin{equation}\label{eq:saturation-trace}
 \rho_{p,q}(n,m)
 \le
 \sup_{\gamma_{p'}(U)\le1}|\operatorname{tr}U|.
\end{equation}

Let $\mathcal G$ be the finite group of signed permutation isometries of
$X$. For every $U:X\to X$,
\[
 \frac1{|\mathcal G|}\sum_{G\in\mathcal G}G^{-1}UG
 =\frac{\operatorname{tr}U}{m}I_X.
\]
Since $\gamma_{p'}$ is an operator-ideal norm, conjugation by elements of
$\mathcal G$ preserves it and convexity gives
\[
 \gamma_{p'}\!\left(\frac{\operatorname{tr}U}{m}I_X\right)
 \le \gamma_{p'}(U).
\]
Consequently,
\[
 \sup_{\gamma_{p'}(U)\le1}|\operatorname{tr}U|
 =\frac{m}{\gamma_{p'}(I_X)}.
\]
Together with \eqref{eq:saturation-trace}, this proves
\begin{equation}\label{eq:saturation-upper}
 \rho_{p,q}(n,m)
 \le \frac{m}{\gamma_{p'}(I_{\ell_q^m})}.
\end{equation}

Conversely, fix a factorization $I_X=BA$ through an $L_{p'}(\mu)$-space.
Since $A(X)$ is finite dimensional, for every $\delta>0$ there is a finite
measurable partition whose conditional expectation $P$ satisfies
\[
 \|(P-I)y\|_{p'}\le\delta\|y\|_{p'}
 \qquad(y\in A(X)).
\]
The range of $P$ is isometric to a finite-dimensional weighted $\ell_{p'}$-space,
and hence to $\ell_{p'}^N$ after rescaling coordinates. For $\delta$ small,
$V:=BPA$ is invertible on $X$ and
\[
 I_X=V^{-1}BPA,
\]
with factorization constant tending to $\|B\|\,\|A\|$ as $\delta\downarrow0$.

Identifying the range of $P$ with $\ell_{p'}^N$, the two legs in
$I_X=V^{-1}BPA$ give, after reciprocal rescaling, operators
$S:X\to\ell_{p'}^N$ and $T:\ell_{p'}^N\to X$ with
$\|S\|,\|T\|\le1$ and
\[
 |\operatorname{tr}(TS)|
 \longrightarrow \frac{m}{\|B\|\,\|A\|}
 \qquad(\delta\downarrow0).
\]
Taking the infimum over the original factorization therefore yields
\[
 \sup_{n\ge1}\rho_{p,q}(n,m)
 =\frac{m}{\gamma_{p'}(I_{\ell_q^m})}.
\]

For $2<q<p'<\infty$, the finite-dimensional $L_{p'}$-factorization estimate
of Gluskin, Pietsch and Puhl \cite{GluskinPietschPuhl1980}, recorded explicitly by
Figiel, Johnson and Schechtman \cite[Section~2]{FigielJohnsonSchechtman1988}, gives

\[
 \gamma_{p'}(I_{\ell_q^m})
 \asymp_{p,q}
 m^{\alpha_{p,q}},
 \qquad
 \alpha_{p,q}
 =\frac{(\frac1q-\frac1{p'})(\frac12-\frac1q)}
 {\frac12-\frac1{p'}}.
\]
Since $1/p'=1-a$ and $1/q=b$, this is exactly the exponent displayed in the
statement. The conclusion follows from \eqref{eq:saturation-upper} and the saturation identity.
The partition argument leaves the dimension of the resulting
finite-dimensional \(\ell_p\)-space uncontrolled. Consequently, a
quantitative relation between \(n\) and \(m\) in the saturation regime
requires additional information.
\end{proof}

\begin{proof}[Proof of Theorem~\ref{thm:B}]
The two-sided bounds are Proposition~\ref{prop:mixedrect}, and the saturation estimate in the interior mixed range is Proposition~\ref{prop:saturation}. It remains to prove the exact statement.

Suppose first that $n\le m$ and $a+b\ge1$. Proposition~\ref{prop:l1q} gives the exact endpoint value
\[
 \rho(\ell_1^n,\ell_q^m)=n^{1-b}.
\]
Comparing $\ell_p^n$ with $\ell_1^n$ therefore yields
\[
 \rho_{p,q}(n,m)\le n^{1-a}n^{1-b}=n^{2-a-b}.
\]
Since $r=n$ and $\gamma_{\rm mix}=2-a-b$, the lower estimate in Proposition~\ref{prop:mixedrect} forces equality.

Suppose next that $n\ge m$ and $a+b\le1$. Corollary~\ref{cor:pinfty} gives the exact endpoint value
\[
 \rho(\ell_p^n,\ell_\infty^m)=m^a.
\]
Comparing $\ell_q^m$ with $\ell_\infty^m$ gives
\[
 \rho_{p,q}(n,m)\le m^b m^a=m^{a+b}.
\]
Now $r=m$ and $\gamma_{\rm mix}=a+b$, so equality again follows from the lower estimate.

The two alternatives above are exactly
\[
 (n-m)(a+b-1)\le0.
\]
When $n=m$ this condition is automatic, and the exact square formula follows. The reversed mixed range follows by Lemma~\ref{lem:duality} and the flip symmetry.
\end{proof}

\section{Operator and XOR-game consequences}\label{sec:consequences}

We next translate the rectangular estimates into operator and XOR-game formulations.

If $E$ and $F$ are finite-dimensional and $T:E^*\to F$ is finite rank, write $\nu(T)$ for its nuclear norm. Under the standard tensor--operator identification,
\[
 \eps(z)=\|T_z\|,
 \qquad
 \pi(z)=\nu(T_z),
\]
and therefore
\begin{equation}\label{eq:operatorrho}
 \rho(E,F)=\sup_{0\neq T:E^*\to F}\frac{\nu(T)}{\|T\|}.
\end{equation}

\begin{corollary}\label{cor:operator}
Let $n,m\in\N$ and assume either $1\le p,q\le2$ or $2\le p,q\le\infty$. With $r=\min\{n,m\}$,
\[
 \sup_{0\neq T:\ell_{p'}^n\to\ell_q^m}\frac{\nu(T)}{\|T\|}
 \asymp_{p,q}
 r^{1/2}\min\{n^{\eta(p)},m^{\eta(q)}\}.
\]
\end{corollary}

\begin{proof}
Combine \eqref{eq:operatorrho} with Theorem~\ref{thm:A}.
\end{proof}

For the real scalar field, Aubrun, Lami, Palazuelos, Szarek and Winter \cite[Eq.~(25)]{AubrunEtAl2020} showed that for fixed local generalized probabilistic theories $A$ and $B$,
\[
 \sup_{AB,G}\frac{\beta_{\rm ALL}(G)}{\beta_{\rm LO}(G)}
 =\rho(V_A,V_B),
\]
where $V_A,V_B$ are their base-norm spaces. They also proved that every finite-dimensional normed space is $2$-isomorphic to a base-norm space.

\begin{corollary}\label{cor:xor}
Assume $\K=\mathbb R$ and either $1\le p,q\le2$ or $2\le p,q\le\infty$. For every $n,m\in\N$ there are local GPTs whose base-norm spaces are at Banach--Mazur distance at most $2$ from $\ell_p^n$ and $\ell_q^m$, respectively, and for which
\[
 \sup_{AB,G}\frac{\beta_{\rm ALL}(G)}{\beta_{\rm LO}(G)}
 \asymp_{p,q}
 r^{1/2}\min\{n^{\eta(p)},m^{\eta(q)}\}.
\]
\end{corollary}

\begin{proof}
Choose the base-norm spaces using \cite{AubrunEtAl2020}. Lemma~\ref{lem:isomorphism}, applied in both variables, changes $\rho$ by at most a universal multiplicative factor. Theorem~\ref{thm:A} gives the stated scale.
\end{proof}

\section{The remaining mixed problem}
\label{sec:questions}

Theorem~\ref{thm:A} gives a complete rectangular formula when the two
exponents lie on the same side of \(2\). In the mixed range,
Theorem~\ref{thm:B} gives the sharp polynomial order whenever
\[
(n-m)\left(\frac1p+\frac1q-1\right)\le0.
\]
The complementary mixed region remains open.

By duality and symmetry, it is enough to consider
\[
1<p<2<q<p',
\qquad
(n-m)\left(\frac1p+\frac1q-1\right)>0.
\]
The diagonal tensor gives the square scale in this region. On the other
hand, Proposition~\ref{prop:saturation} identifies the saturation scale
\[
m^{1-\alpha_{p,q}}
\]
after taking the supremum over the first dimension.

The unresolved point is a finite-dimensional description of the transition
between these two scales. More precisely, the proof of
Proposition~\ref{prop:saturation} approximates a factorization through
an \(L_{p'}\)-space by factorizations through finite-dimensional
\(\ell_{p'}\)-spaces, while the dimension required in this approximation remains
uncontrolled. Determining that dimension amounts to finding a quantitative
threshold relating \(n\) and \(m\).

In particular, whether the saturation order is attained whenever \(n\)
exceeds a fixed power of \(m\) remains open.

A natural question is whether, throughout the remaining mixed region,
\[
\rho_{p,q}(n,m)
\asymp_{p,q}
\min\left\{
n^{1-1/p}m^{1-1/q},
\,
m^{1-\alpha_{p,q}}
\right\},
\]
with the corresponding statement in the reversed mixed range obtained
by duality and by interchanging the two tensor factors.

\section*{Acknowledgements}
We are grateful to Andreas Defant for kindly drawing our attention to the papers of Bonet, Defant, Peris and Ramanujan and of Defant and Mascioni, and for helpful correspondence. These references were important in clarifying the historical context of the square problem and its connection with the classical theory of $\eps$--$\pi$ continuous operators.

\section*{Declarations}

\noindent\textbf{Competing interests.} The authors declare that they have no competing interests.

\noindent\textbf{Funding.} D. N\'u\~nez-Alarc\'on, D. Pellegrino, and J. Santos are supported by Grants No. 406457/2023-9 (CNPq/MCTI No. 10/2023) and No. 403964/2024-5 (MCTI/CNPq No. 16/2024), both from the Conselho Nacional de Desenvolvimento Cient\'ifico e Tecnol\'ogico (CNPq, Brazil). In addition, D. Pellegrino is supported by Grant No. 305807/2025-0.

\noindent\textbf{Data availability.} No data were generated or analyzed in this study.

\noindent\textbf{Declaration of generative AI and AI-assisted technologies in the manuscript preparation process.} During the preparation of this work, the authors used OpenAI's ChatGPT to assist with language editing, organization, and presentation of the manuscript, including checks of expository consistency and references. The manuscript was reviewed and edited by the authors, who take full responsibility for its content.

\medskip
\noindent\textbf{Author addresses.}

\noindent Daniel N\'u\~nez-Alarc\'on: Departamento de Matem\'aticas, Facultad de Ciencias, Universidad Nacional de Colombia, Sede Bogot\'a, Bogot\'a, Colombia.\\
\texttt{dnuneza@unal.edu.co}

\noindent Daniel Marinho Pellegrino: Departamento de Matem\'atica, Centro de Ci\^encias Exatas e da Natureza, Universidade Federal da Para\'iba, Jo\~ao Pessoa, PB, Brazil.\\
\texttt{daniel.pellegrino@academico.ufpb.br}

\noindent Joedson Silva dos Santos: Departamento de Matem\'atica, Centro de Ci\^encias Exatas e da Natureza, Universidade Federal da Para\'iba, Jo\~ao Pessoa, PB, Brazil.\\
\texttt{joedson.santos@academico.ufpb.br}


\begin{thebibliography}{99}

\bibitem{AubrunEtAl2020}
G. Aubrun, L. Lami, C. Palazuelos, S. J. Szarek and A. Winter,
Universal gaps for XOR games from estimates on tensor norm ratios,
\emph{Comm. Math. Phys.} \textbf{375} (2020), 679--724.

\bibitem{AubrunMullerHermes2025}
G. Aubrun and A. M\"uller-Hermes,
Limit formulas for norms of tensor power operators,
\emph{J. Funct. Anal.} \textbf{289} (2025), Paper No. 111113.

\bibitem{AubrunMullerHermes2026}
G. Aubrun and A. M\"uller-Hermes,
Asymptotic tensor powers of Banach spaces,
\emph{Ann. Inst. Fourier (Grenoble)} \textbf{76} (2026), 1341--1368.

\bibitem{BonetDefantPerisRamanujan1994}
J. Bonet, A. Defant, A. Peris and M. S. Ramanujan,
Coincidence of topologies on tensor products of K\"othe echelon spaces,
\emph{Studia Math.} \textbf{111} (1994), no. 3, 263--281.

\bibitem{DefantFloret}
A. Defant and K. Floret,
\emph{Tensor Norms and Operator Ideals},
North-Holland Mathematics Studies, Vol. 176, North-Holland, Amsterdam, 1993.

\bibitem{DefantMascioni1990}
A. Defant and V. Mascioni,
Limit orders of $\eps$--$\pi$ continuous operators,
\emph{Math. Nachr.} \textbf{145} (1990), 337--344.

\bibitem{DefantMichels2006}
A. Defant and C. Michels,
Norms of tensor product identities,
\emph{Note Mat.} \textbf{25} (2006), no. 1, 129--166.


\bibitem{FigielJohnsonSchechtman1988}
T. Figiel, W. B. Johnson and G. Schechtman,
Random sign embeddings from $\ell_r^n$, $2<r<\infty$,
\emph{Proc. Amer. Math. Soc.} \textbf{102} (1988), no. 1, 102--106.


\bibitem{GluskinPietschPuhl1980}
E. D. Gluskin, A. Pietsch and J. Puhl,
A generalization of Khintchine's inequality and its application in the theory of operator ideals,
\emph{Studia Math.} \textbf{67} (1980), no. 2, 149--155.

\bibitem{GuptaMisraRay}
R. Gupta, G. Misra and S. K. Ray,
On a variant of the Grothendieck inequality and estimates on tensor product norms,
arXiv:2305.13270v2 [math.FA], 2023.


\bibitem{LatalaStrzeleckaJFA2025}
R. Lata{\l}a and M. Strzelecka,
Operator $\ell_p\to\ell_q$ norms of random matrices with iid entries,
\emph{J. Funct. Anal.} \textbf{288} (2025), no. 3, Paper No. 110720, 38 pp.

\bibitem{PisierGT}
G. Pisier,
Grothendieck's theorem, past and present,
\emph{Bull. Amer. Math. Soc. (N.S.)} \textbf{49} (2012), 237--323.

\bibitem{Ryan}
R. A. Ryan,
\emph{Introduction to Tensor Products of Banach Spaces},
Springer Monographs in Mathematics, Springer-Verlag, London, 2002.

\end{thebibliography}
\end{document}